\documentclass[12pt]{amsart}

\usepackage{latexsym,amssymb} 
\usepackage{amsmath,amsfonts,amsthm}
\input xypic

\newtheorem{theorem}{Theorem}[section]
\newtheorem{proposition}[theorem]{Proposition}

\newtheorem{lemma}[theorem]{Lemma}

\begin{document}

\title[Generalized heaps]{Generalized heaps, inverse semigroups and Morita equivalence}

\author{M.~V.~Lawson}
\address{Department of Mathematics and the Maxwell Institute for Mathematical Sciences\\
Heriot-Watt University\\
Riccarton\\
Edinburgh~EH14~4AS\\
Scotland}
\email{M.V.Lawson@ma.hw.ac.uk}



\begin{abstract} Inverse semigroups are the abstract counterparts of pseudogroups of transformations.
The abstract counterparts of atlases in differential geometry are what Wagner termed `generalized heaps'.
These are sets equipped with a ternary operation satisfying certain axioms.
We prove that there is a bijective correspondence between generalized heaps and the equivalence bimodules, defined by Steinberg.
Such equivalence bimodules are used to define the Morita equivalence of inverse semigroups.
This paper therefore shows that the Morita equivalence of inverse semigroups is determined by Wagner's generalized heaps.
 \end{abstract}

\maketitle

%%%%%%%%%%%%%%%%%%%%%%%%%%%%%%%%%%%%%%%%%%%%%%%%%%%%%%%%%%%%%%%%%%%%%%%%%%%%%%%%%%%%%%%%%%%%%%%%%%%%%%%%%%%%%%%%%%%%%%%%%%%%%%%%%%%%%%%%%%%%%%%%%%%%%%%%%%%%%%%%%%%
\section{Introduction}

Inverse semigroups originated as the algebraic counterparts of pseudogroups of transformations in differential geometry.
We refer the reader to \cite{L2} for more information on this and for all undefined terms from inverse semigroup theory.
V.~V.~Wagner\footnote{This name is usually transliterated as `Vagner' in the literature,
but we understand that `Wagner' was his preferred transliteration.}, one of the founders of the field, was one of the few who continued to seek inspiration from this source. 
In differential geometry, pseudogroups are usually not studied on their own but in combination with the notion of an atlas.
Just as Wagner defined inverse semigroups to be the algebraic versions of pseudogroups, 
so too he defined a class of structures, called {\em generalized heaps}\footnote{Sometimes the original Russian word used for heap,
`groud', is used   }, to be the algebraic versions of atlases.
To understand his definition, we briefly recall the essence of the usual definition of an atlas from differential geometry.

Given two spaces $X$ and $Y$, an {\em atlas} 
$A$ from $X$ to $Y$ is a set of partial bijections
such that the union of their domains is $X$ and the union of their images is $Y$.
The set $T = A^{-1}A$ is a collection of partial bijections defined on $X$, and $S = AA^{-1}$ is a collection of
partial bijections defined on $Y$.
For example, if $X$ and $Y = \mathbb{R}^{n}$ are topological spaces, $A$ consists of homeomorphisms,
and $S$ is the pseudogroup of all smooth maps defined between open subsets of $Y$ then the atlas $A$ defines
the structure of a differential manifold on $X$.
Other such local structures can be defined in a similar way.

This concrete notion of an atlas can be made algebraic.
Observe that if $x,y,z \in A$, an atlas, then also $xy^{-1}z \in A$, as long as $A$ is sufficiently large.
Define a ternary operation on the set $A$ by 
$$\{xyz\} = xy^{-1}z.$$
One may seek to axiomatise the resulting structure and relate it back, in the spirit of Cayley's theorem, to the original concrete notion of an atlas.
This was what Wagner did.
The result is a set $A$ equipped with a ternary operation $(x,y,z) \mapsto \{xyz\}$ satisfying the following axioms:
\begin{description}

\item[{\rm (A1)}] $\{xxx\} = x$.

\item[{\rm (A2)}] $\{\{x_{1}x_{2}x_{3}\}x_{4}x_{5}\} = \{x_{1}\{x_{4}x_{3}x_{2}\}x_{5}\} = \{x_{1}x_{2}\{x_{3}x_{4}x_{5}\}\}$.

\item[{\rm (A3)}] $\{xx\{yyz\}\} = \{yy\{xxz\}\}$.

\item[{\rm (A4)}] $\{\{zxx\}yy\} = \{\{zyy\}xx\}$.

\end{description}
Such structures are called {\em generalized heaps}.
The axiomatization above is taken from \cite{BS2} a translation of \cite{BS1}.

This aspect of the work by Wagner and his school did not become well known outside of Eastern Europe for a variety of reasons:
mathematically, generalized heaps are unusual in being based on a ternary operation rather than a binary one;
more substantively, the theory of generalized heaps appears to be tangential to the main theory of inverse semigroups,
and gives the appearance of being nothing more than a generalization for generalization's sake;
finally, the theory was developed at a time when political tensions between East and West impeded the dissemination of ideas.
Whatever the reasons, although Wagner is one of the founding fathers of inverse semigroup theory, 
the details of this particular aspect of his work have been largely forgotten.
One exception was in synthetic differential geometry \cite{AK1}.

However, in recent years, references to Wagner's work in this area have started to reappear \cite{BK,Gr}.
This paper has the goal of doing more by showing that generalized heaps are in fact a central component of inverse semigroup theory:
they are precisely what is needed to define the Morita equivalence of inverse semigroups.

Specifically, we prove that the equivalence bimodules used to define Morita equivalence of inverse semigroups in \cite{S}
are in bijective correspondence with generalized heaps.
In other words, equivalence bimodules are to classical atlases as inverse semigroups are to pseudogroups.
Far from being tangential to inverse semigroup theory or a generalization for generalization's sake,
generalized heaps are the mathematical devices which witness a Morita equivalence between two inverse semigroups.
This has added interest in the light of the connection between inverse semigroups and \'etendues \cite{F,F1,L4}.

The material in Lemma~3.2 is due to Wagner and is described in \cite{W3}.
The main tool we use is that of a {\em pregroupoid} which is due to Anders Kock \cite{AK2} with origins in \cite{AK1}
combined with the Ehresmann-Schein-Nambooripad theorem described in \cite{L2}.
A small sample of Wagner's work in this area can be found in \cite{W1,W2,W3,W4}.
In addition to whatever mathematical merit this paper may have,
we also hope that it will serve as inspiration for a deeper study of Wagner's {\em oeuvre}:
not just as an important ingredient in the history of semigroup theory,
but also as part of the development of differential geometry in the twentieth century. \\

\noindent
{\bf Acknowledgements }This paper arose in the course of writing \cite{FLS};
I am grateful to my co-authors, Jonathon Funk and Benjamin Steinberg, for their inspiring ideas on this topic.
I would also like to thank Grigori Zhitomirskii for sending me a copy of \cite{W3}.
Finally, I am grateful to the referee for suggesting the additional references \cite{HS}, \cite{BS} and \cite{SS},
and that I include the example at the end of the paper.\\

%%%%%%%%%%%%%%%%%%%%%%%%%%%%%%%%%%%%%%%%%%%%%%%%%%%%%%%%%%%%%%%%%%%%%%%%%%%%%%%%%%%%%%%%%%%%%%%%%%%%%%%%%%%%%%%%%%%%%%%%%%%%%%%%%%%%%%%%%%%%%%%%%%%%%%%%%%%%
\section{From equivalence bimodules to generalized heaps}\setcounter{theorem}{0}

In this section, we prove the easy direction and show that each equivalence bimodule gives rise to a generalized heap.

If $S$ is an inverse semigroup and $X$ a set then a {\em left action} of $S$ on $X$ is a function
$S \times X \rightarrow X$, denoted by $(s,x) \mapsto s \cdot x$, such that $(st) \cdot x = s \cdot (t \cdot x)$.
We shall usually write $sx$ rather than $s \cdot x$.
Such an action is said to be {\em unitary} if $SX = X$.
If there is a unitary action of $S$ on $X$ then we say that  $X$ is a {\em left $S$-module}.
Right actions and right $S$-modules are defined dually.
If $S$ and $T$ are inverse semigroups and $X$ is both a left $S$-module and a right $T$-module
and $(s \cdot x) \cdot t = s \cdot (x \cdot t)$ for all $s \in S$, $t \in T$ and $x \in X$
then we say that $X$ is an {\em $(S,T)$-bimodule}.

The following definition is due to Steinberg \cite{S}.
Let $S$ and $T$ be inverse semigroups.
An {\em equivalence bimodule} for $S$ and $T$ consists of an $(S,T)$-bimodule equipped with surjective functions
$$\langle -,- \rangle \colon \: X \times X \rightarrow S
\text{ and }
[-,-] \colon \:X \times X \rightarrow T$$
such that the following axioms hold, where $x,y,z \in X$ and $s \in S$ and $t \in T$:
\begin{description}

\item[{\rm (MC1)}] $\langle sx,y \rangle = s\langle x,y \rangle$.

\item[{\rm (MC2)}] $\langle y,x \rangle = \langle x,y \rangle^{-1}$.

\item[{\rm (MC3)}] $\langle x,x \rangle x = x$.

\item[{\rm (MC4)}] $[x,yt] = [x,y]t$.

\item[{\rm (MC5)}] $[x,y] = [y,x]^{-1}$.

\item[{\rm (MC6)}] $x[x,x] = x$.

\item[{\rm (MC7)}] $\langle x,y \rangle z = x [y,z]$.

\end{description}
If $S$ and $T$ have such an equivalence bimodule they are said to be {\em strongly Morita equivalent}. 
It was shown in \cite{FLS} that this is the correct definition of Morita equivalence for inverse semigroups:
in particular, it dovetails well with the Morita theory of topological groupoids and $C^{\ast}$-algebras \cite{P}.

Each inverse semigroup $S$ gives rise to an equivalence bimodule in a simple way.
We put $X = S$ and we let $S$ act on $X$ on the left and right by left and right multiplication.
Both actions are clearly unitary and so $X$ is an $(S,S)$-bimodule.
If we define $\langle -,- \rangle \colon \: X \times X \rightarrow S$ by $\langle x,y \rangle = xy^{-1}$ and
$[-,-] \colon \:X \times X \rightarrow T$ by $[x,y] = x^{-1}y$,
then  $(S,S,X,\langle -,- \rangle,[-,-])$ with these definitions is an equivalence bimodule.
We shall denote this equivalence bimodule by $EB(S)$.

%%%%%%%%%%%%%%%%%%%%%%%%%%%%%%%%%%%%%%%%%%%%%%%%%%%%%%%%%%%%%%%%%%%%%%%%%%%%%%%%%%%%%%%%%%%%%%%%%%%%%%%%%%%%%%%%%%%%%%%%%%%%
We need to recall some results from \cite{S} which we prove for the sake of completeness.

\begin{lemma} Let $(S,T,X,\langle -,- \rangle,[-,-])$ be an equivalence bimodule.
\begin{enumerate}

\item For each $x \in X$ both $\langle x, x \rangle$ and $[x,x]$ are idempotents.

\item Define the relation $\leq_{S}$ on $X$ by 
$$x \leq_{S} y \Leftrightarrow x = \langle x,x\rangle y.$$
Then $\leq_{S}$ is a partial order on $X$.

\item  Define the relation $\leq_{T}$ on $X$ by 
$$x \leq_{T} y \Leftrightarrow x = y[x,x].$$
Then $\leq_{T}$ is a partial order on $X$.

\item The two orders $\leq_{S}$ and $\leq_{T}$ coincide.

\end{enumerate}
\end{lemma}
\begin{proof}

(1) We prove that $[x,x]$ is an idempotent; the fact that $\langle x, x \rangle$ is an idempotent follows by symmetry.
We have that
$$[x,x][x,x] = [x x[x,x]] = [x,x]$$
by (MC4) and (MC6).

(2) We observe first that if $x = ey$ where $e$ is an idempotent then $x \leq_{S} y$. 
This is because $ex = x$ and so
$$\langle x, x \rangle y 
= \langle ex, ex \rangle y 
= e \langle x, x \rangle e y 
= \langle x,x \rangle ey 
=\langle x,x \rangle x = x$$ 
where we have used (MC1), (MC2) and the fact that $e$ and $\langle x,x \rangle$ are idempotents.
It follows that $x \leq_{S} y$, as claimed.

By (MC3), the relation $\leq_{S}$ is reflexive.
Suppose that $x \leq_{S} y$ and $y \leq_{S} x$.
Then $x = \langle x, x \rangle y$ and $y = \langle y,y \rangle x$.
Thus
$$x = \langle x, x \rangle y 
= \langle x, x \rangle \langle y,y \rangle x
= \langle y, y \rangle \langle x,x \rangle x
=  \langle y, y \rangle x
= y$$
using the fact by (1) that $\langle x, x \rangle$ and  $\langle y,y \rangle$ are idempotents
and that idempotents commute since $S$ is an inverse semigroup.
We have therefore shown that $\leq_{S}$ is antisymmetric.
Finally, let $x \leq_{S} y$ and $y \leq_{S} z$.
Then $x = \langle x, x\rangle \langle y, y \rangle z$.
But this is just $x = ey$ where $e$ is some idempotent and so $x \leq_{S} y$ by our first observation.

(3) This follows by symmetry from (2) above.

(4) Suppose that $x \leq_{S} y$. Then $x = \langle x,x\rangle y$. Put $e = \langle x,x\rangle$ to simplify notation in the calculation that follows.
We calculate
$$y[x,x] 
= y[ey,ey] 
= \langle y, ey \rangle ey 
=  e\langle y, y \rangle ey 
= \langle y, y \rangle ey 
= ey = x$$
where we have used (MC7), (MC1) and the fact that $e$ and $\langle y, y \rangle$ are idempotents.
\end{proof}

In the light of (4) above, we denote $\leq_{S}$ and $\leq_{T}$ by $\leq$.

%%%%%%%%%%%%%%%%%%%%%%%%%%%%%%%%%%%%%%%%%%%%%%%%%%%%%%%%%%%%%%%%%%%%%%%%%%%%%%%%%%%%%%%%%%%%%%%%%%%%%%%%%%%%%%%%%%%%%%%%%%%%%%%%%%%%%%%%%%%%%%%%%%
We may now prove our first main result.

\begin{proposition} Let $(S,T,X,\langle -,- \rangle,[-,-])$ be an equivalence bimodule.
On the set $X$ define a ternary operation
$$\{xyz\} = \langle x,y \rangle z.$$
Then $(X,\{\})$ is a generalized heap.
\end{proposition}
\begin{proof} 

(A1) holds. We have that $\{xxx\} = \langle x, x\rangle x = x$ by (MC3).

(A2) holds. By definition and (MC1), we have that
$$\{\{x_{1}x_{2}x_{3}\}x_{4}x_{5}\} = \langle \langle x_{1}, x_{2} \rangle x_{3}, x_{4} \rangle x_{5}
= \langle x_{1},x_{2} \rangle \langle x_{3}, x_{4} \rangle x_{5};$$
by (MC1) and (MC2), we have that
$$\{x_{1}, \{x_{4}x_{3}x_{2}\}x_{5}\} = \langle x_{1}, \langle x_{4}, x_{3} \rangle x_{2} \rangle x_{5}
= \langle x_{1}, x_{2} \rangle \langle x_{4}, x_{3} \rangle^{-1}x_{5}$$
which is equal to
$$\langle x_{1},x_{2} \rangle \langle x_{3}, x_{4} \rangle x_{5};$$
and, finally,
$$\{x_{1}x_{2}\{x_{3}x_{4}x_{5}\}\}
=
\langle x_{1}, x_{2} \rangle \langle x_{3}, x_{4} \rangle x_{5}$$
where we have used the associativity of the action.

%(A3) holds. Suppose that $\{xyx\} = x$ and $\{yxy\} = y$.
%Then $\langle x, y \rangle x = x$ and $\langle y, x \rangle y = y$.
%Observe that 
%$$\langle x,x \rangle = \langle  \langle x,y \rangle x, x\rangle
%= \langle x, y \rangle \langle x, x\rangle.$$
%Thus $\langle x, x \rangle \leq \langle x, y \rangle$
%since $\langle x, x \rangle$ is an idempotent by Lemma~2.1.
%It follows that $\langle x,x \rangle \leq \langle y, x \rangle$ also.
%Now
%$$x = \langle x,x \rangle x = \langle x, x\rangle \langle y, x \rangle x
%= \langle x, x \rangle y[x,x]$$
%by (MC7).
%However, this implies that $x \leq y$ using Lemma~2.1.
%A dual argument shows that $y \leq x$ and so $x = y$, as required.

(A3) holds. We have that
$$\{xx \{yyz\} \} = \langle x, x \rangle \langle y, y \rangle z$$
whereas
$$\{yy\{xxz\}\} = \langle y,y \rangle \langle x, x\rangle z.$$
These two elements are equal because $\langle x, x \rangle$ and $\langle y, y \rangle$ are idempotents and so commute.

(A4) holds. We have that
$$\{ \{zxx\}yy\} = z[x,x][y,y]$$
whereas
$$\{ \{zyy\}xx\} = z[y,y][x,x]$$
using (MC1) and (MC7).
These two elements are equal because $[x,x]$ and $[y,y]$ are idempotents and so commute.
\end{proof}

We have seen that from each inverse semigroup $S$ we can construct an equivalence bimodule $EB(S)$.
The generalized heap constructed from $EB(S)$ by Proposition~2.2
is simply the set $S$ equipped with the ternary operation defined by $\{xyz\} = xy^{-1}z$.
We denote this generalized heap by $GH(S)$.

The obvious question is whether every generalized heap determines an equivalence bimodule.
This we shall answer, in the affirmative, in the remainder of this paper.

%%%%%%%%%%%%%%%%%%%%%%%%%%%%%%%%%%%%%%%%%%%%%%%%%%%%%%%%%%%%%%%%%%%%%%%%%%%%%%%%%%%%%%%%%%%%%%%%%%%%%%%%%%%%%%%%%%
\section{From generalized heaps to equivalence bimodules}

This direction is more complex and will be carried out in a series of steps.

Given a generalized heap $(X,\{ \})$ our goal is to construct an equivalence bimodule.
This requires us to construct two inverse semigroups.
As a first step, we shall construct two meet semilattices.

Recall that a {\em band} is a semigroup in which every element is an idempotent.
A band is {\em left normal} if it satisfies the law $xyz = xzy$, and it is {\em right normal} if it satisfies the law $xyz = yxz$.
A commutative band is just a {\em semilattice}.
Our first result is well-known \cite{H}.

\begin{lemma} \mbox{}
\begin{enumerate}

\item Let $S$ be a right normal band.
Then the minimum semilattice congruence on $S$ is Green's relation  $\mathcal{R}$.

\item Let $S$ be a left normal band.
Then the minimum semilattice congruence on $S$ is Green's relation  $\mathcal{L}$.

\end{enumerate}
\end{lemma}
\begin{proof} In any band, the minimum semilattice congruence is Green's relation $\mathcal{D}$.
Suppose the band is right normal and that $e \mathcal{L} f$.
Then $ef = e$ and $fe = f$.
But then $efe = e$ and $efe = fe$.
Thus $e = f$.
It follows that $\mathcal{L}$ is the equality relation and so, since $\mathcal{D} = \mathcal{L} \circ \mathcal{R}$,
it follows that the minimum semilattice congruence is $\mathcal{R}$.
\end{proof}

Our next result tells us that from a generalized heap we can construct left and right normal bands and so,
by the above lemma, we may construct two semilattices.
 
\begin{lemma}
Let $X$ be a generalized heap.
\begin{enumerate}

\item Define the binary operation $\circ$ on $X$ by $x \circ y = \{xxy\}$.
Then $(X,\circ)$ is a right normal band.
Put $E = X^{\circ}/\mathcal{R}$, a semilattice, and denote the natural map from $X$ to $E$ by $p$.

\item Define the binary operation $\bullet$ on $X$ by $x \bullet y = \{xyy\}$.
Then $(X,\bullet)$ is a left normal band.
Put $F = X^{\bullet}/\mathcal{L}$, a semilattice, and denote the natural map from $X$ to $F$ by $q$.

\item A mixed associativity law holds
$$(x \circ y) \bullet z = x \circ (y \bullet z)$$
for all $x,y,z \in X$.

\end{enumerate}
\end{lemma}
\begin{proof}
We prove (1); the proof of (2) follows by symmetry.
The fact that every element is an idempotent follows by (A1).
We prove associativity.
By definition
$$(x \circ y) \circ z = \{\{xxy\}\{xxy\}z\}.$$ 
But
$$\{\{xxy\}\{xxy\}z\} = \{ \{ \{xxy\}yx\}xz\}$$
by (A2).
By (A2) and (A3) we have that
$$
\{ \{ \{xxy\}yx\}xz\}
= \{ \{xx\{yyx\}\}xz\}
= \{\{yy\{xxx\}\}xz\}.
$$
By (A1) this is equal to
$$\{\{yyx\}xz\}.$$
Finally we use (A2) and (A3) to get
$$\{yy\{xxz\}\} 
= \{xx\{yyz\}\}
=
x \circ (y \circ z),
$$
as required.
We have thefore proved that we have a band.
To show that we have a right normal band observe that
$$x \circ y \circ z = \{xx\{yyz\}\} = \{yy\{xxz\}\} = y \circ x \circ z$$
using (A3).

(3) We have that
$$x \circ (y \bullet z)
=
\{xx \{ y \bullet z) \}
= 
\{ xx \{ y zz \}\}
=
\{\{xxy\}zz\}
=
(x \circ y) \bullet z.
$$

\end{proof}

We write $X^{\circ}$ and $X^{\bullet}$ when we wish to regard the set $X$ with respect to each of these two binary operations.
By the above 
$$p(x) = p(y) \Leftrightarrow  x = y \circ x \mbox{ and } y = x \circ y \Leftrightarrow x \mathcal{R} y \text{ in } X^{\circ}$$
and
$$q(x) = q(y) \Leftrightarrow  x = x \bullet y \mbox{ and } y = y \bullet x \Leftrightarrow x \mathcal{L} y \text{ in } X^{\bullet}.$$
The elements of the generalized heap $X$ should be regarded as arrows
$$p(x) \stackrel{x}{\longleftarrow} q(x).$$
We therefore have the following diagram
$$
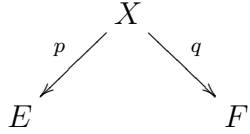
\xymatrix
{
&  & X \ar[dl]_{p} \ar[dr]^{q}   &\\
&E & &F
}$$

Left normal and right normal bands are called {\em restrictive semigroups of the first and second kind} in \cite{BS}.
To explain this terminology, we apply Lemma~3.2 to the generalized heap $GH(S)$ constructed from the inverse semigroup $S$.
The operation $\circ$ is therefore the binary operation defined on the set $S$ by $x \circ y = xx^{-1}y$.
Regarding an inverse semigroup as a semigroup of partial bijections, 
this operation restricts the partial bijection $y$ to the range of $x$.
It is therefore an extension of the corestriction operation defined in the inductive groupoid associated with every inverse semigroup \cite{L2}.

%%%%%%%%%%%%%%%%%%%%%%%%%%%%%%%%%%%%%%%%%%%%%%%%%%%%%%%%%%%%%%%%%%%%%%%%%%%%%%%%%%%%%%%%%%%%%%%%%%%%%%%%%%%%%%%%%%%%%%%%%%%%%%%%%%%%%%%%%%%%
The next step is to show that from each generalized heap we can construct a pregroupoid in the sense of Kock \cite{AK1,AK2}.
Let $X$ be a set equipped with a {\em partially defined} ternary operation $\{\}$, 
and surjections $p \colon \: X \rightarrow E$ and $q \colon \: X \rightarrow F$
such that $\{xyz\}$ is defined if and only if $q(x) = q(y)$ and $p(y) = p(z)$ 
and such that the following axioms hold:
\begin{description}

\item[{\rm (PG1)}] $p(\{xyz\}) = p(x)$ and $q(\{xyz\}) = q(z)$.

\item[{\rm (PG2)}] $\{xxz\} = z$ and $\{yxx\} = y$.

\item[{\rm (PG3)}] $\{vy \{ yxz \}\} = \{vxz\}$ and $\{\{yxz\}zw\} = \{yxw\}$.

\end{description}
Then we call $(X,\{\},p,q)$ a {\em pregroupoid}.

If $X$ is a generalized heap, then we define the {\em restricted product} on $X$ to be the ternary operation
restricted to those triples $(x,y,z)$ where $q(x) = q(y)$ and $p(y) = p(z)$.
Such triples should be regarded in the following way
$$
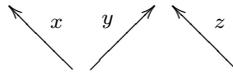
\xymatrix{
& & & & \\
& &\ar[ul]_{x} \ar[ur]^{y} & &\ar[ul]_{z}
}$$
This should be read from right-to-left and interpreted as `$xy^{-1}z$'. 

\begin{proposition} Let $X$ be a generalized heap and let $p \colon \: X \rightarrow E$ and $q \colon \: X \rightarrow F$ be defined as above.
Then with respect to the restricted product $(X,\{\},p,q)$ is a pregroupoid.
\end{proposition}
\begin{proof}
(PG1) holds. Suppose that $q(x) = q(y)$ and $p(y) = p(z)$.
Thus $x = \{xyy\}$, $y = \{yxx\}$ and $y = \{zzy\}$ and $z = \{yyz\}$.
We have that
$$\{xx\{xyz\}\} = \{\{xxx\}yz\} = \{xyz\}$$
and
$$\{\{xyz\}\{xyz\}x\} = \{\{xyz\}z\{yxx\}\} = \{\{xyz\}zy\} = \{xy\{zzy\}\} = \{xyy\} = x.$$
Thus $p(\{xyz\}) = p(x)$.

We also have that
$$\{\{xyz\}zz\} = \{xy\{zzz\}\} = \{xyz\}$$
and
$$\{z\{xyz\}\{xyz\}\} = \{\{zzy\}x\{xyz\}\} = \{yx\{xyz\}\} = \{\{yxx\}yz\} = \{yyz\} = z.$$
Thus $q(\{xyz\}) = q(z)$.

(PG2) holds. Both of these follow immediately from the definitions

(PG3) holds. We have that
$$\{vy\{yxz\}\} = \{\{vyy\}xz\} = \{vxz\}.$$

Similarly, we have that
$$\{\{yxz\}zw\} = \{yx\{zzw\}\} = \{yxw\}.$$

\end{proof}

%%%%%%%%%%%%%%%%%%%%%%%%%%%%%%%%%%%%%%%%%%%%%%%%%%%%%%%%%%%%%%%%%%%%%%%%%%%%%%%%%%%%%%%%%%%%%%%%%%%%%%%%%%%%%%%%%%%%%%%%%%%%%%%%%%%%%%%%%%%%%%%%%%%%%%%%
We now follow Kock \cite{AK2} and use this pregroupoid structure to construct two groupoids that we denote by $XX^{-1}$ and $X^{-1}X$.
We define $X^{-1}X$; the definition of $XX^{-1}$ is obtained dually.
Let 
$$XpX = \{(x,y) \in X \times X \colon \: p(x) = p(y) \}.$$
This is just the groupoid/equivalence relation determined by the relation
$\mathcal{R}$ on the semigroup $X^{\circ}$
We regard the elements of this set as diagrams
$$
\xymatrix{
& & &\\
&q(x)\ar[ur]^{x} & &q(y)\ar[ul]_{y}
}$$
which we informally think of as `$x^{-1}y$'.
On the set $XpX$ define the following relation
$$(x,y) \preceq (u,v) \Leftrightarrow x = x \bullet u, \, y = y \bullet v, \, y = \{xuv\}.$$

\begin{lemma} 
The relation $\preceq$ is a preorder and the equivalence relation $\equiv$ it determines is given by
$$(x,y) \equiv (u,v) \Leftrightarrow q(x) = q(u), q(y) = q(v), y = \{xuv\}.$$
\end{lemma}
\begin{proof}
The fact that $\preceq$ is reflexive follows from the fact that $X^{\bullet}$ is a band and the fact that $p(x) = p(y)$.
Let $(x,y) \preceq (u,v)$ and $(u,v) \preceq (w,z)$.
Then $x = x \bullet u$, $y = y \bullet v$, $y = \{xuv\}$,
and
$u = u \bullet w$, $v = v \bullet z$, $v = \{uwz\}$.
It is immediate that $x = x \bullet w$ and $y = y \bullet z$
and a simple calculation to show that $y = \{xwz \}$.
Thus $\preceq$ is transitive.

It is immediate that if $(x,y) \equiv (u,v)$ then $q(x) = q(u), q(y) = q(v), y = \{xuv\}$.
We prove the converse by showing that from $q(x) = q(u), q(y) = q(v), y = \{xuv\}$
we may deduce that $v = \{uxy\}$.
But 
$$\{uxy\} = \{ux\{xuv\}\} = \{ \{uxx\}uv\} = \{uuv\} = v$$
since $u = \{uxx\}$ and $v = \{uuv\}$.
\end{proof}

We denote the equivalence class containing the pair $(x,y)$ by $x^{-1}y$ and the set of equivalence classes by $X^{-1}X$.
The element $x^{-1}y$ should be regarded as an arrow 
$$\xymatrix{
&q(x) &  &q(y) \ar[ll]_{x^{-1}y}
}
$$

Define a partial binary operation on $X^{-1}X$ by
$$x^{-1}y \cdot u^{-1}v = x^{-1}\{yuv\}$$
if and only if $q(y) = q(u)$.

\begin{lemma} With the above definitions we have the following:
\begin{enumerate}

\item $X^{-1}X$ is a groupoid equipped with an order defined by 
$$x^{-1}y \leq u^{-1}v \Leftrightarrow x = x \bullet u \mbox{ and } y = \{xuv\}$$
whose set of identities can be identified with the semilattice $F$.

\item $XX^{-1}$ is a groupoid equipped with an order defined by
$$xy^{-1} \leq uv^{-1} \Leftrightarrow y = v \circ y  \mbox{ and } x = \{uvy\}$$
whose set of identities can be identified with the semilattice $E$

\end{enumerate}
\end{lemma}
\begin{proof} We prove (1); the proof of (2) follows by symmetry.
The order is well-defined on the equivalence classes since it is precisely the one coming
from the preorder used to define that relation.
The fact that $X^{-1}X$ with the partial binary operation is a groupoid follows from \cite{AK2}.

We check that the identities of the groupoid $X^{-1}X$ can be identified with the semilattice $F$.
Identities have the form $x^{-1}x$.
Observe that $x^{-1}x = y^{-1}y$ iff $q(x) = q(y)$.
Thus there is a bijection between the identities of $X^{-1}X$ and the semilattice $F$.
Also $x^{-1}x \leq y^{-1}y$ iff $x = x \bullet y$.
It follows that the partially ordered set of identities of $X^{-1}X$
is order-isomorphic with the meet semilattice $F$.
\end{proof}

%%%%%%%%%%%%%%%%%%%%%%%%%%%%%%%%%%%%%%%%%%%%%%%%%%%%%%%%%%%%%%%%%%%%%%%%%%%%%%%%%%%%%%%%%%%%%%%%%%%%%%%%%%%%%%%%%%%%%%%%%%%%%%%%%%%%%
Both groupoids $X^{-1}X$ and $XX^{-1}$ are equipped with partial orders.
We next show that $X^{-1}X$ and $XX^{-1}$ are in fact inverse semigroups by using the theory of ordered groupoids:
specifically the Ehresmann-Schein-Nambooripad theorem \cite{L2}.

\begin{proposition} \mbox{}
\begin{enumerate}

\item $X^{-1}X$ is an inductive groupoid with pseudoproduct given by
$$x^{-1}y \otimes u^{-1}v = \{\{yuu\}yx\}^{-1}\{yuv\}.$$
The semilattice of idempotents of $X^{-1}X$ is isomorphic to the semilattice $F$.

\item $XX^{-1}$ is an inductive groupoid with pseudoproduct given by
$$xy^{-1} \otimes uv^{-1} = \{xyu\}\{vu\{yyu\}\}^{-1}.$$
The semilattice of idempotents of $XX^{-1}$ is isomorphic to the semilattice $E$.

\end{enumerate}
\end{proposition}
\begin{proof}
We prove (1); the proof of (2) follows by symmetry.\\

\noindent
(i) $x^{-1}y \leq u^{-1}v \Rightarrow y^{-1}x \leq v^{-1}u$.\\
This follows from the fact that
$$\{yvu\}
= \{ \{xuv\}vu \}
= \{xuu\}
= x.$$

\noindent
(ii) $x_{1}^{-1}y_{1} \leq u_{1}^{-1}v_{1}$ and $x_{2}^{-1}y_{2} \leq u_{2}^{-1}v_{2}$
implies that
$x_{1}^{-1}y_{1}x_{2}^{-1}y_{2}
\leq
u_{1}^{-1}v_{1}u_{2}^{-1}v_{2}$
where the products are groupoid products.\\
We have that $x_{1} = x_{1} \bullet u_{1}$ and $y_{1} = \{x_{1}u_{1}v_{1}\}$
and $x_{2} = x_{2} \bullet u_{2}$ and $y_{2} = \{x_{2}u_{2}v_{2}\}$.
We shall prove that $x_{1}^{-1}\{y_{1}x_{2}y_{2}\} \leq u_{1}^{-1}\{v_{1}u_{2}v_{2}\}$.
We have $x_{1} = x_{1} \bullet u_{1}$.
Now $q(v_{1}) = q(u_{2})$ and so $v_{1} = \{v_{1}u_{2}u_{2}\}$.
Thus 
$$\{y_{1}u_{2}u_{2}\} = \{ \{x_{1}u_{1}v_{1}\}u_{2}u_{2}\} = \{x_{1}u_{1} \{v_{1}u_{2}u_{2}\}\}
= \{x_{1}u_{1}v_{1}\} = y_{1}.$$
We now calculate
$$\{x_{1}u_{1}\{v_{1}u_{2}v_{2}\}\}
=
\{ \{x_{1}u_{1}v_{1}\}u_{2}v_{2}\}
= 
\{y_{1}u_{2}v_{2}\}
=
\{y_{1}u_{2} \{u_{2}x_{2}y_{2}\}\}$$
this is equal to
$$\{ \{y_{1}u_{2}u_{2}\}x_{2}y_{2}\}$$
which is just
$\{y_{1}x_{2}y_{2}\}$,
as required.\\

\noindent
(iii) We now construct corestrictions.\\
Let $z^{-1}z \leq x^{-1}x$.
Define 
$$(z^{-1}z \mid x^{-1}y) = z^{-1}\{zxy\}.$$
We prove that this is a corestriction.
It is easy to check that $z^{-1}\{zxy\} \leq  x^{-1}y$.
Let $u^{-1}v \leq x^{-1}y$ where $u^{-1}u = z^{-1}z$.
Then 
$$\{uz \{zxy\} \} = \{ \{uzz\}xy\} = \{uxy\} = v.$$
We have therefore proved uniqueness.\\

It follows that $X^{-1}X$ is an ordered groupoid \cite{L2}.
It is inductive because the partially ordered set of identities forms a semilattice 
isomorphic to $F$.
Using the restriction and corestriction operations we can now calculate the pseudoproduct.
We get
$$x^{-1}y \otimes u^{-1}v 
=
\{\{yuu\}yx\}^{-1}\{yuu\} \cdot \{yuu\}^{-1}\{\{yu\}uv\}
$$
which quickly simplifies to
$$\{\{yuu\}yx\}^{-1}\{yuv\}.$$
\end{proof}

From now on we shall denote the pseudoproducts by concatenation.

%%%%%%%%%%%%%%%%%%%%%%%%%%%%%%%%%%%%%%%%%%%%%%%%%%%%%%%%%%%%%%%%%%%%%%%%%%%%%%%%%%%%%%%%%%%%%%%%%%%%%%%%%%%%%%%%%%%%%%%%%%%%%%%%%%%%%%%%%%
With the above notation and results, we now make the following definitions:
\begin{itemize}

\item Define $\langle -,- \rangle \colon \: X \times X \rightarrow XX^{-1}$ by 
$$\langle x,y \rangle = \{xyy\}\{yxx\}^{-1} = (x \bullet y)(y \bullet x)^{-1}.$$

\item Define $[-,-]\colon \: X \times X \rightarrow X^{-1}X$ by 
$$[x,y] = \{yyx\}^{-1}\{xxy\} = (y \circ x)^{-1}(x \circ y).$$
 
\item Define $XX^{-1} \times X \rightarrow X$ by
$xy^{-1} \cdot z = \{xyz\}$.

\item Define $X \times X^{-1}X \rightarrow X$ by
$x \cdot y^{-1}z = \{xyz\}$.

\end{itemize}

\begin{lemma} \mbox{}
\begin{enumerate}

\item $\langle -,- \rangle \colon \: X \times X \rightarrow XX^{-1}$ is well-defined and surjective.

\item $[-,-]\colon \: X \times X \rightarrow X^{-1}X$ is well-defined and surjective.

\item $XX^{-1} \times X \rightarrow X$ is well-defined.

\item $X \times X^{-1}X \rightarrow X$ is well-defined.

\item Axioms (MC2),(MC3),(MC5),(MC6) and (MC7) hold. 

\end{enumerate}
\end{lemma}
\begin{proof}
(1).  We prove that $q(\{xyy\}) = q(\{yxx\})$.
We calculate one part of the proof
$$\{ \{xyy\}\{yxx\}\{yxx\}\}
=
\{\{\{xyy\}xx\}y\{yxx\}\}
=
\{\{xyy\}y\{yxx\}\}$$
which is equal to
$$
\{x\{yyy\}\{yxx\}\}
=
\{xy\{yxx\}\}
=
\{\{xyy\}xx\}
=
\{xxx\}yy\}
=
\{xyy\}.$$
It remains to show that this map is surjective.
Let $xy^{-1} \in XX^{-1}$.
Then $q(x) = q(y)$.
Thus $x = \{xyy\}$ and $y = \{yxx\}$.
It follows that $\langle x,y \rangle  = \{xyy\}\{yxx\}^{-1} = xy^{-1}$,
as required.

(2). We prove that $p(\{yyx\}) = p(\{xxy\})$.
We calculate one part of the proof
$$\{\{yyx\}\{yyx\}\{xxy\}
=
\{\{\{yyx\}x\}\{xxy\}\}
=
\{\{\{yyx\}xy\}y\{xxy\}\}$$
which is equal to
$$\{\{yyx\}x\{yy\{xxy\}\}\}
=
\{\{yyx\}x\{xxy\}\}
=
\{yy\{xxy\}\}
=
\{xxy\}.$$
It remains to show that this map is surjective.
Let $x^{-1}y \in X^{-1}X$.
Then by assumption $p(x) = p(y)$.
Thus $x = \{yyx\}$ and $y = \{xxy\}$.
It follows that $[x,y] = \{yyx\}^{-1}\{xxy\} = x^{-1}y$,
as required.

(3)  We have to show that this operation is well-defined; this is similar to the proof of (4) below.

(4). We have to show that this operation is well-defined.
Let $y^{-1}z = u^{-1}v$.
We have that
$$x \cdot y^{-1}z = \{xyz\}$$
and
$$x \cdot u^{-1}v = \{xuv\}.$$
By assumption
$z = \{yuv\}$.
Thus
$$\{xyz\} = \{xy\{yuv\}\} = \{x \{uyy\}v\}.$$
But $q(y) = q(u)$ and so $u = \{uyy\}$.
Thus $\{xyz\} = \{xuv\}$, as required.

(5). (M2) By construction $\langle x, y \rangle$ and $\langle y, x \rangle$ are groupoid inverses of each other.

(M3) By definition $\langle x,x \rangle x = \{xxx\} = x$.

(M5) By construction $[x,y]$ and $[y,x]$ are groupoid inverses of each other.

(M6) By definition $x[x,x] = \{xxx\} = x$.

(M7) By definition 
$$\langle x, y \rangle z = \{ \{xyy\}\{yxx\}z\}$$
which quickly simplifies to $\{xyz\}$.
By definition
$$x[y,z] = \{x \{zzy\}\{yyz\}\}$$
which quickly simplifies to $\{xyz\}$.
\end{proof}

%%%%%%%%%%%%%%%%%%%%%%%%%%%%%%%%%%%%%%%%%%%%%%%%%%%%%%%%%%%%%%%%%%%%%%%%%%%%%%%%%%%%%%%%%%%%%%%%%%%%%%%%%%%%%%%%%%%%%%%%%%%%%%%%%%%%%%%%%%%%%%%%%%%%%%%%%
\begin{proposition} $X$ is a $(XX^{-1},X^{-1}X)$-bimodule,
and (MC1) and (MC4) hold. 
\end{proposition}
\begin{proof}
We show that $X$ is a left $XX^{-1}$-module.
We have that 
$$(xy^{-1}uv^{-1}) \cdot z = \{\{xyu\}\{vu\{yyu\}\}z\}$$
whereas
$$xy^{-1} \cdot (uv^{-1} \cdot z) = \{xy \{uvz\}\}.$$
But
$$\{\{xyu\}\{vu\{yyu\}\}z\}
=
\{\{xy \{ u \{yyu\}u\}\}vz\}
= 
\{\{\{xy\{yyu\}\}vz\}$$
which is equal to
$$\{\{xyy\}yu\}
=
\{\{xyu\}vz\}
=
\{xy\{uvz\}\}.$$
Thus $X$ is a left $XX^{-1}$-module.
A dual argument shows that $X$ is a right $X^{-1}X$-module.

To show that it is a bimodule we calculate
$(xy^{-1} \cdot z) \cdot u^{-1}v$
and
$xy^{-1} \cdot (z \cdot u^{-1}v)$.
But these are equal by (A2).

(MC1) holds. We calculate $\langle xy^{-1} \cdot u,v \rangle$ and $xy^{-1}\langle u, v \rangle$.
Now 
$$\langle xy^{-1} \cdot u,v \rangle
=
\langle \{xyu\}, v \rangle
=
\{\{xyu\}vv\}\{v\{xyu\}\{xyu\}\}^{-1}$$
and
$$xy^{-1}\langle u, v \rangle
=
xy^{-1} \otimes \{uvv\}\{vuu\}^{-1}
=
\{xy\{uvv\}\}\{\{vuu\}\{uvv\}\{yy\{uvv\}\}\}^{-1}.$$
To show that these two elements are equal, we need to show that the `negative' parts of these two expressions are equal.

We have that
$$\{\{vuu\}\{uvv\}\{yy\{uvv\}\}\}
=
\{(v \bullet u)(u \bullet v)(y \circ (u \bullet v))\}
=
\{ ((v \bullet u) \bullet v)u(y \circ (u \bullet v)) \}$$
where we have used the fact, Lemma~3.2, that $X^{\bullet}$ is a left normal band,
and this is equal to
 $$
\{(v \bullet u)u(y \circ (u \bullet v)\}
=
\{\{vuy\}yu\}\bullet v.$$

We also have that
$$v \bullet \{xyu \} = v \bullet v \bullet \{xyu \} = v \bullet \{xyu\} \bullet v$$
again using the fact that $X^{\bullet}$ is a left normal band.
This is just
$$\{ \{v\{xyu\}\{xyu\}\} vv\}.$$
But
$$\{v\{xyu\}\{xyu\}\}
=
\{vu\{yx\{xyu\}\}\}
=
\{vu\{\{yxx\}yu\}\}
=
\{vu\{yyu\}\}
=
\{\{vuy\}yu\}$$
using the fact that $x = \{yyx\}$.
Thus
$$\{v\{xyu\}\{xyu\}\}
=
\{\{\{vuy\}yu\}vv\}.$$
We have therefore shown that (MC1) holds.

The fact that (MC4) holds follows by a dual argument.
\end{proof}

Combining the above results we get the following.

\begin{theorem} With each generalized heap $(X,\{\})$ we can associate an equivalence bimodule
$$(XX^{-1},X^{-1}X,X,\langle -,-\rangle,[-,-]).$$
\end{theorem}

%%%%%%%%%%%%%%%%%%%%%%%%%%%%%%%%%%%%%%%%%%%%%%%%%%%%%%%%%%%%%%%%%%%%%%%%%%%%%%%%%%%%%%%%%%%%%%%%%%%%%%%%%%%%%%%%%%%%%%%%%%%%%%%%%%%%%%%%%%%%%%
\section{Back and forth}

It remains to show that the two constructions we have described are essentially inverses of each other.
The following lemma is part of Proposition~2.3 of \cite{S}.

\begin{lemma} In an equivalence bimodule, the following hold.
\begin{enumerate}

\item $[x,y][z,w] = [x,  \langle y, z \rangle w ]$.

\item $[xt,y] = t^{-1}[x,y]$.

\item $[sx,y] = [x,s^{-1}y]$, $[x,sy] = [s^{-1}x,y]$.

\end{enumerate}
\end{lemma}

Let $(X,\{ \})$ be a generalized heap.
By Theorem~3.9, there is an equivalence bimodule constructed from it
and by Proposition~2.2 there is, in turn, a heap constructed from this bimodule.
The underlying set of this heap is also $X$ and its ternary operation is 
$$\{\{xyy\}\{yxx\}z\}.$$
But by applying the axioms of a generalized heap this quickly simplifies to $\{xyz\}$.
We have therefore shown that if we start with a generalized heap, 
construct the corresponding equivalence bimodule, 
and then construct the generalized heap from that,
we arrive back where we started.
Thus we need only prove the following.

\begin{proposition}
Let
$(S,T,X,\langle -,-\rangle,[-,-])$
be an equivalence bimodule and let
$(XX^{-1},X^{-1}X,X,\langle -,- \rangle_{1},[-,-]_{1})$.
be the equivalence bimodule that arises after successively applying our two constructions.
Then the two equivalence bimodules are isomorphic.
\end{proposition}
\begin{proof}
We show first that $X^{-1}X$ is isomorphic to $T$.

Define a map $X^{-1}X \rightarrow S$ by $x^{-1}y \mapsto [x,y]$.
This map is well-defined,
for suppose that $x^{-1}y = u^{-1}v$.
Then 
$y = \{xuv\} = \langle x,u \rangle v$
and
$u = \{uxx\} = \langle u. x \rangle x$.
We calculate
$$[x,y]
= 
[x, \langle x,u \rangle v]
=
[\langle u, x \rangle x,v]
=
[u,v]$$
using Lemma~4.1(3).

Next we show that this map is injective.
Let $x^{-1}y,u^{-1}v \in X^{-1}X$ and suppose that $[x,y] = [u,v]$.
Then 
$$y 
= \langle x, x \rangle y
= x[x,y]
= x[u,v]
= \langle x,u \rangle v.$$
Thus $y = \{xuv\}$.
Next we show that $q(x) = q(u)$ and $q(y) = q(v)$.
We have that
$$x 
= \langle y, y \rangle x
= y[y,x]
= y[v,u]$$ 
because $[y,x] = [v,u]$.
But $[v,u][u,u] = [v,u]$ by (MC2).
Thus $x[u,u] = x$ by Lemma~4.1(1).
Hence $x = \{xuu\}$.
We may similarly show that $u = \{uxx\}$.
Thus $q(x) = q(u)$.
A similar argument shows that $q(y) = q(v)$.
Thus the map is injective.

We now show that the map is surjective.
Let $s \in S$.
Then by assumption there exists $(x,y) \in X \times X$ such that $[x,y] = s$.
Observe that
$$p(\{yyx\}) = p(\{xxy\}).$$
Thus
$\{yyx\}^{-1}\{xxy\} \in X^{-1}X$.
Then
$$
[\{yyx\},\{xxy\}]
= [\langle y,y \rangle x, \langle x, x \rangle y]
= [x,y] = s
$$ 
by Lemma~4.1(2).

It remains to show that this function is a homomorphism.
By definition
$$x^{-1}y \otimes u^{-1}v 
=
\{\{yuu\}yx\}^{-1} \{yuv \}.$$
This maps to the element
$$[\{\{yuu\}yx\},\{yuv \}]$$
which is just
$$[ \langle  \langle y,u \rangle u   , y \rangle x, \langle y, u \rangle v ].$$
But
$$\langle y, u \rangle v = y [u,v].$$
Thus the above gives
$$[ \langle  \langle y,u \rangle u   , y \rangle x, y][u,v].$$
But 
$$[ \langle  \langle y,u \rangle u   , y \rangle x, y]
=
[x, \langle y, \langle y,u\rangle u \rangle y ]
=
[x, y[\langle y,u \rangle u,y]]                                        
=
[x,y][\langle y,u \rangle u,y]$$
which is equal to
$$
[x,y][y[u,u],y]
=
[x,y][u,u][y,y],$$
using Lemma~4.1(2).
We have therefore proved that the image of
$x^{-1}y \otimes u^{-1}v$ is equal to $[x,y][u,v]$
as required.

Thus
$\alpha \colon \: X^{-1}X \rightarrow S$ given by $\alpha (x^{-1}y) = [x,y]$ is an isomorphism of semigroups.
A dual argument shows that
$\beta \colon \: XX^{-1} \rightarrow T$ given by $\beta (xy^{-1}) = \langle x, y \rangle$ is an isomorphism of semigroups.

We now show that the actions are isomorphic.
By definition
$$x \cdot y^{-1}z = \{xyz\} = \langle x, y \rangle z = x[y,z] = x \alpha (y^{-1}z).$$
A dual argument holds for the action of $XX^{-1}$ on $X$.

Finally, we compare $[-,-]_{1}$ and $[-,-]$.
By definition
$$[x,y]_{1} 
= \{yyx\}^{-1}\{xxy\} 
= [\{yyx\},\{xxy\}]
= [\langle y,y \rangle x, \langle x, x \rangle y]
= [x,y]$$
 using Proposition~2.3 of \cite{S}.
The dual argument compares $\langle -,- \rangle_{1}$ and $\langle -,- \rangle$.
\end{proof}

We conclude this paper with an example.
Let $S$ be an inverse semigroup.
We have constructed an equivalence bimodule $EB(S)$ from $S$ whose corresponding
generalized heap $GH(S)$ is simply the set $S$ equipped with
the ternary operation $(x,y,z) \mapsto xy^{-1}z$.
The theory above tells us that if we start with $GH(S)$ then its aassociated equivalence bimodule will be isomorphic to $EB(S)$;
it is an interesting exercise to work through the details in this particular case:
for example, $p(x) = p(y)$ if and only if $xx^{-1} = yy^{-1}$ and, significantly, the inverse semigroups $S^{-1}S$ and $SS^{-1}$ are both isomorphic to $S$.
There is an interesting conclusion to be deduced from this example.
Suppose that we only know the set $S$ and the ternary operation $(x,y,z) \mapsto xy^{-1}z$, can we recover $S$?
The answer from the constructions of this paper is `yes'.

%%%%%%%%%%%%%%%%%%%%%%%%%%%%%%%%%%%%%%%%%%%%%%%%%%%%%%%%%%%%%%%%%%%%%%%%%%%%%%%%%%%%%%%%%%%%%%%%%%%%%%%%%%%%%%%%%%%%%%%%%%%%%%%%%%%%%%%%

\bibliographystyle{plain}

\end{document}